\documentclass{scrartcl}
\usepackage{amsmath,amsthm,amssymb}

\usepackage[utf8]{inputenc}
\usepackage[T1]{fontenc}
\usepackage{lmodern}
\usepackage{mathdots}

\newtheorem{theorem}{Theorem}
\newtheorem{lemma}[theorem]{Lemma}
\newtheorem{corollary}[theorem]{Corollary}
\theoremstyle{definition}

\newtheorem{example}[theorem]{Example}

\usepackage{pgfplots,pgfplotstable}
\usetikzlibrary{plotmarks}

\usepackage{colonequals}
\usepackage{microtype}
\usepackage{hyperref}

\DeclareMathOperator\rev{rev}
\newcommand{\m}[1]{\begin{bmatrix}#1\end{bmatrix}}
\DeclareMathOperator\diag{diag}

\newcommand{\abs}[1]{\left\vert#1\right\vert}
\newcommand{\norm}[1]{\left\Vert#1\right\Vert}
\newcommand{\col}[1]{{#1}^{\text{col}}}
\newcommand{\row}[1]{{#1}^{\text{row}}}

\newcommand{\val}[2]{\left. #1 \right|_{#2}}

\title{A duality relation for matrix pencils with application to linearizations}
\author{Federico Poloni\footnote{Technische Universit\"at Berlin, Institut f\"ur Mathematik. Strasse des 17. Juni 136, 10623 Berlin. \texttt{poloni@math.tu-berlin.de}. Supported by the Alexander von Humboldt foundation.}}

\begin{document}
\maketitle

\begin{abstract}
The aim of this paper is twofold. First, we introduce a new class of linearizations, based on the generalization of a construction used in polynomial algebra to find the zeros of a system of (scalar) polynomial equations. We show that one specific linearization in this class, which is constructed naturally from the QR factorization of the matrix obtained by stacking the coefficients of $A(x)$, has good conditioning and stability properties.

Moreover, while analyzing this class, we introduce a general technique to derive new linearizations from existing ones. This technique generalizes some ad-hoc arguments used in dealing with the existing linearization classes, and can hopefully be used to derive a simpler and more general theory of linearizations. This technique relates linearizations to \emph{pencil arithmetic}, a technique used in solving matrix equations that allows to extend some algebraic operations from matrix to matrix pencils.
\end{abstract}

\section{Introduction}
The most used technique for finding the eigenvalues of a matrix polynomial is converting it to a linear problem. Given a matrix polynomial $A(x)\colonequals \sum_{i=0}^d A_d x^d \in \mathbb{C}^{n,n}[x]$, a degree-1 matrix polynomial (also called \emph{matrix pencil}) $L(x)\in \mathbb{C}^{dn,dn}[x]$ is a \emph{linearization} if there are $E(x),F(x) \in \mathbb{C}^{dn,dn}[x]$ such that
\begin{equation}\label{deflin}
 L(x)=E(x)\diag(A(x),I_{(d-1)n})F(x),
\end{equation}
and $\det E(x),\det F(x)$ are nonzero constants. Linearizations have the same eigenvalues and Jordan structure as the starting matrix polynomial  \cite{GohLR-book}. Several different methods to construct such a pencil exist; we recall here the more common ones.
\begin{description}
 \item[Companion forms] \cite{GohLR-book} The well-known \emph{companion matrix} of a scalar polynomial generalizes easily to matrix polynomials. An example of this construction is $C_0-C_1 x$, where
\begin{align}\label{compform}
 C_0\colonequals& \m{A_0\\ & I \\ & & I\\ & & & \ddots \\ & &  & & I}, & C_1\colonequals \m{-A_1 & I_{n} &  \\ -A_2& 0 & I_n & &\\ \vdots& & \ddots & \ddots  & \\ -A_{d-1} & & & 0 & I_n\\ -A_d& & & & 0}.
\end{align}
In literature several variants of this form appear: the polynomial may be reversed, so that $A_d$ (instead of $A_0$) appears alone on a coefficient, and the $A_i$ can appear in the last block column instead of the first, or in the top or bottom block row.
\item[Companion forms in different polynomial bases] \cite{AmiCR09} The companion form can be adapted using in its construction different (orthogonal) polynomial families instead of the usual monomial basis, yielding for instance pencils in the form $C_0-\widetilde{C_1} x$, where
\[
 \widetilde{C_1}\colonequals\begin{bmatrix}
                  -A_1 + \beta_0 A_0 & \gamma_1 I\\
                  -A_2 + \alpha_0 A_0 & \beta_1 I & \gamma_2 I\\
                  -A_3 & \alpha_1 I& \beta_2 I& \gamma_3 I\\
                  -A_4 & & \alpha_2 I& \beta_3 I& \ddots \\
                  \vdots & & & \ddots & \ddots & \gamma_{d-1} I\\
                  -A_{d-1} & & & & \alpha_{d-2}I & \beta_{d-1}I
                 \end{bmatrix}
\]
has a tridiagonal pattern related to the associated three-term recurrence.
\item[Vector spaces of linearizations] \cite{MacMMM06} Another generalization of \eqref{compform} is considering the vector space $\mathbb{L}_2$ of pencils $L_0-L_1 x$ that satisfy
\begin{equation}\label{rowshifted}
 \begin{bmatrix}
  I_{dn}\\0_{d\times dn}
 \end{bmatrix}
  L_0 -
 \begin{bmatrix}
  0_{d\times dn}\\I_{dn}
 \end{bmatrix}
  L_1=
  \begin{bmatrix}
   A_0\\A_1\\\vdots \\ A_d
  \end{bmatrix}
  \left(v^* \otimes I_n\right)
\end{equation}
for some $v\in\mathbb{C}^{d}$. The operation on the left-hand side is called \emph{row-shifted sum} in \cite{MacMMM06}. A second vector space $\mathbb{L}_1$ is constructed with a similar relation, generalizing another companion form in which the coefficients of the polynomial appear on a row. Pencils in the intersection $\mathbb{DL}\colonequals \mathbb{L}_1 \cap \mathbb{L}_2$ have many interesting properties; for any regular matrix polynomial, most of them are linearizations, and there is a simple test to tell when and to characterize their condition number.
\item[Fiedler pencils] \cite{AntV04}
Define the matrices $F_i\colonequals\diag(I_{(i-1)n},G_i,I_{n(d-i-1)})$, for $i=1,2,\dots,d$, where
\[
 G_i\colonequals \begin{bmatrix}A_i & I_n\\ I_n & 0_{n\times n}\end{bmatrix}.
\]
For each permutation $\sigma$ of $(1,2,\dots,d)$, the pencil $C_0-\prod_{i=1}^d F_{\sigma(i)} x$ is a linearization; in particular, the two permutations $\sigma(i)=i$ and $\sigma(d+1-i)$ yield companion forms, so this construction can again be regarded a generalization, though in a different sense, of the companion form. All the pencils obtained in this way can be laid out explicitly as a simple combination of blocks $A_i$, $I$ and zeros. Several additional generalizations of Fiedler pencils exist \cite{AntV04,VolA11}.
\end{description}
Not all linearizations preserve the Jordan structure of the eigenvalue $\infty$ \cite{GohKL88,LanP06}; to ensure that (\emph{strong linearization property}), we have to require additionally that $\rev L(x)$ is a linearization for $\rev A(x)$, where the operator $\rev$ is defined by
\[
 \rev \sum_{i=0}^d M_i x^i \colonequals \sum_{i=0}^d M_{d-i} x^i.
\]
Similarly, not all linearizations extend seamlessly to singular polynomials. Usually, the minimal indices of the singular part of the matrix polynomials are modified; therefore, in order to reconstruct the singular structure of $A(x)$, we need to analyze how it is related to that of $L(x)$ \cite{DetDM09}.

Other additional properties are often desirable in a linearization. For instance, when the starting polynomial is structured (i.e., symplectic, or even, or Hermitian), we would like our linearization to have the same property \cite{MacMMM06-goodVibrations}. Moreover, in addition to the eigenvalues, it is useful to be able to recover the eigenvectors of $A(x)$ with little additional work (preferably none at all) from those of $L(x)$ \cite{MacMMM06,BueDD11}. Last but not least, the condition number of an eigenvalue in $L(x)$ can be higher than the corresponding one in $A(x)$ \cite{HigMT06}, leading to instability in the computation; it would be good to be able to control this condition number growth.

The aim of this paper is twofold. First, we introduce a new class of linearizations, based on the generalization of a construction used in polynomial algebra to find the zeros of a system of (scalar) polynomial equations. We show that one specific linearization in this class, which is constructed naturally from the QR factorization of the matrix obtained by stacking the coefficients of $A(x)$, has good conditioning and stability properties.

Moreover, while analyzing this class, we introduce a general technique to derive new linearizations from existing ones. This technique generalizes some ad-hoc arguments used in dealing with the existing linearization classes, and can hopefully be used to derive a simpler and more general theory of linearizations. This technique relates linearizations to \emph{pencil arithmetic} \cite{BenB06}, a technique used in solving matrix equations that allows to extend some algebraic operations from matrix to matrix pencils.

\section{Basics}
For each $\mathbb{C}$-vector space $M$ (and in particular we are interested in $M=\mathbb{C}^n$ or $M=\mathbb{C}^{n,n}$) and for each degree $d\in \mathbb{N}$ we define
\[
M_d[\lambda,\mu]\colonequals \{a_d \lambda^d + a_{d-1} \lambda^{d-1}\mu +\dots + a_0 \mu^d  : a_i \in M \},
\]
the space of degree-$d$ homogeneous matrix polynomials in two unknowns with coefficients in $M$. We identify a single-variable matrix polynomials) $A(x)$ of degree $d$ with its \emph{projectivization} $A(\lambda/\mu)\mu^d$. This makes it easier to deal with the eigenvalue $x=\infty$ (which becomes $\lambda=1,\mu=0$); on the other hand, a purely projective theory of linearizations
seems to be elusive, as a projective analogue of \eqref{deflin} and of the strong linearization property does not exist, up to our knowledge.

Two matrix polynomials $P(x),Q(x)\in\mathbb{C}^{n,n}[x]$, are called \emph{equivalent} if there exist $E(x),F(x)\in\mathbb{C}^{n,n}[x]$ such that $E(x)P(x)F(x)=Q(x)$ and $\det E(x)$, $\det P(x)$ are nonzero constant. They are called \emph{strongly equivalent} if one can find constant matrices $E(x)\equiv E\in\mathbb{C}^{n,n}$, $F(x)\equiv F\in\mathbb{C}^{n,n}$ realizing the equivalence. Notice that both are equivalence relations. Any pencil equivalent to a linearization is a linearization; any pencil strongly equivalent to a strong linearization is a strong linearization.

A matrix polynomial is called \emph{regular}, or \emph{nonsingular}, if $\det A(x)$ is not identically zero. Nonsingularity of a matrix polynomial $A(x)=\sum_{x=0}^d A_i x^i$ implies that
\[
 \col{A}\colonequals\m{A_0\\A_1\\ \vdots \\ A_d} \text{ and } \row{A}\colonequals\m{A_0 & A_1 & \dotsm & A_d }
\]
have respectively full column and row rank; the viceversa does not hold, though \cite{GohLR-book}.

\begin{theorem}\label{thm:companion}
Let a matrix polynomial $A(\lambda,\mu)\colonequals \sum_{i=0}^d A_i \lambda^i \mu^{d-i} \in \mathbb{C}^{n,n}_d[\lambda,\mu]$ be given, and let $C_0,C_1$ be as in \eqref{compform}.
Then,
\begin{enumerate}
 \item the pencil $C(\lambda,\mu)\colonequals\lambda C_1 -\mu C_0$, known as \emph{second companion form}, is a strong linearization of $A(\lambda,\mu)$.
 \item Let $x$ and $y^*$ be right and left eigenvectors of $A(\lambda,\mu)$ with eigenvalue $(\lambda,\mu)$, i.e., $A(\lambda,\mu)x=0$ and $y^* A(\lambda,\mu)=0$. Then, the vectors
\begin{align*}
 \hat{x}=&\m{x\\ (A_1+\frac{\mu}{\lambda} A_0)x\\(A_2+\frac{\mu}{\lambda}A_1+\frac{\mu^2}{\lambda^2} A_0)x\\ \vdots\\
 (A_{d-1} + \frac{\mu}{\lambda} A_{d-2} + \dots + \frac{\mu^{d-1}}{\lambda^{d-1}} A_0)x
 }\\
 =&\m{x\\ -(\frac{\lambda}{\mu}A_2+\frac{\lambda^2}{\mu^2}A_3 + \dotso + \frac{\lambda^{d-1}}{\mu^{d-1}}A_d)x \\ \vdots \\ -(\frac{\lambda}{\mu}A_{d-1}+\frac{\lambda^2}{\mu^2}A_d)x \\ -\frac{\lambda}{\mu}A_dx},\\
 \hat{y}^*=&\m{\mu^{d-1}y^* & \mu^{d-2}\lambda y^* & \dotsm & \lambda^{d-1} y^*}
\end{align*}
are right and left eigenvectors of $\lambda C_1 -\mu C_0$, respectively. Notice that the first expression for $\hat{x}$ is defined for $\lambda\neq 0$, the second for $\mu\neq 0$, and they coincide when they are both defined.
\end{enumerate}
\end{theorem}

\begin{lemma}[Bases completion lemma]\label{bcl}
Let $W\in\mathbb{C}^{m,m+n}$ with full row rank and $A\in\mathbb{C}^{m+n,n}$ with full row rank such that $WA=0$. Then, one can find $V\in\mathbb{C}^{n,n+m}$, $B\in\mathbb{C}^{n+m,m}$ such that
\[
 \m{V\\W}\m{A & B}=I_{m+n}.
\]
\end{lemma}
\begin{proof}
 Let us complete the columns of $A$ and $W^*$ to a basis of $\mathbb{C}^{m+n}$, i.e., choose $\hat{B}\in\mathbb{C}^{n+m,m}$ and $\hat{V}\in\mathbb{C}^{n,m+n}$ such that $\m{A & \hat{B}}$ and $\m{\hat{V}^* &W^*}$ are nonsingular. Define
 \[
  \m{R_{11} & R_{12}\\ 0 & R_{22}}\colonequals\m{\hat{V}\\W}\m{A & \hat{B}};
 \]
 $R_{11},R_{22}$ are invertible, since the right-hand side is the product of two invertible matrices. Finally, define $B$ and $V$ from the relations
 \begin{align*}
  \m{A & B}=&\m{A & \hat{B}}\m{I & -R_{12}R_{22}^{-1}\\ 0 & R_{22}^{-1}}, &
  \m{V\\W}=&\m{R_{11}^{-1} & 0\\ 0 & I}\m{\hat{V}\\W}.
  \qedhere
 \end{align*}
\end{proof}

\section{Polynomial rings and Möller--Stetter algorithm}
In commutative algebra, a polynomial ideal $I=(f_1,\dots,f_k) \subseteq\mathbb{C}[x_1,\dots,x_d] $ is called \emph{zero-dimensional} if the system
\[
 \begin{cases}
  f_1(x_1,\dots,x_d)=0\\
  \hfill\vdots\hfill{} \\
  f_k(x_1,\dots,x_d)=0
 \end{cases}
\]
has only a finite number of solutions, or, equivalently \cite[Finiteness Theorem, Section 2.5]{CoxLittleOSheaUsingAlgebraicGeometry}, if the space $\mathbb{C}[x_1,\dots,x_d] / I$ is finite-dimensional. When this holds, we have the following result.
\begin{theorem}[Möller--Stetter]
 Let $I$ be a zero-dimensional ideal, and $f\in \mathbb{C}[x_1,\dots,x_d] / I$. Consider the multiplication map $M_g: \mathbb{C}[x_1,\dots,x_d] / I \to \mathbb{C}[x_1,\dots,x_d] / I$ defined as  $f \mapsto fg$. When a basis of $\mathbb{C}[x_1,\dots,x_d] / I$ is chosen, this map is represented by a matrix. Its eigenvalues (counted with multiplicity) are the values of $g(\hat{x}_1,\dots,\hat{x}_d)$ at every solution (counted with multiplicity) $\hat{x}_1,\dots,\hat{x}_d$ of the polynomial equations.
\end{theorem}
Solutions of the polynomial system can then be found by solving the associated eigenproblem \cite{CoxLittleOSheaUsingAlgebraicGeometry}. Even though this stronger condition is usually not of interest in the commutative algebra applications, in the 1-variable case it is possible to prove that $M_x$ is a linearization of the single polynomial equation $f_1=0$ that generates the ideal (there is a single one, as $\mathbb{C}[x]$ is a PID). In fact, when the monomial basis is chosen, the multiplication operator is the companion matrix of $f_1$.

We aim to generalize this result to polynomial matrices and to the projective setting in order to produce linearizations. Given a projective matrix polynomial $A(\lambda,\mu)\colonequals \sum_{i=0}^d A_i \lambda^i \mu^{d-i} \in \mathbb{C}^{n,n}_d[\lambda,\mu]$, we define the analogous of the ``generated ideal'' $(A)$ as
\[
 (A)\colonequals\{v\in \mathbb{C}^n_d[\lambda,\mu] : v(\lambda,\mu)=A(\lambda,\mu) u \text{ for some $u\in \mathbb{C}^n$}\}.
\]
Note that $(A)$ is a vector space. When the polynomial is regular, it has dimension exactly $n$ (proof: the map $u\mapsto A(\lambda,\mu)u$ must have trivial kernel); in fact, to prove this we only need the weaker condition that the $A_i$ have trivial common (right) kernel.

A natural generalization of the multiplication map $M_x$ used above is considering the maps
\begin{align*}
M_\lambda:& f(\lambda,\mu)\mapsto \lambda f(\lambda,\mu), \\
M_\mu:& f(\lambda,\mu)\mapsto \mu f(\lambda,\mu).
\end{align*}
Notice, however, that these operations increase the degree of the homogeneous polynomials they operate on, thus it is not immediate to see what we should choose as their domain and codomain to replace $\mathbb{C}[x_1,\dots,x_d]/I$. After some attempts, one ends up with the idea of considering the quotiented operators
\begin{equation}\label{qmultop}
\begin{gathered}
 \bar{M}_\lambda, \bar{M}_\mu: \mathbb{C}^n_{d-1}[\lambda,\mu] \to \mathbb{C}^n_{d}[\lambda,\mu] / (A),\\
 \bar{M}_\lambda: v(\lambda,\mu) \mapsto \lambda v(\lambda,\mu) + (A),\\
 \bar{M}_\mu: v(\lambda,\mu) \mapsto \mu v(\lambda,\mu) + (A),
\end{gathered}
\end{equation}
and the hope that $\mu \bar{M}_\lambda-\lambda \bar{M}_\mu$ is a linearization. Such a hope is well founded, as we see in the following theorem.
\begin{theorem}\label{thm:multop}
Let $A(\lambda,\mu)=A_d \lambda^d + A_{d-1} \lambda^{d-1} \mu + \dotso + A_0 \mu^d$ be any matrix polynomial such that $(A)$ has dimension $n$. Let $W\in\mathbb{C}^{dn,(d+1)n}$ be any matrix with full row rank such that $W\col{A}=0$.
\begin{enumerate}
 \item \label{item1}
The matrices
\begin{align}\label{defws}
W_0\colonequals W&\m{0_{n,dn}\\I_{dn}}, & W_1\colonequals W&\m{I_{dn}\\0_{n,dn}}
\end{align}
represent, using suitable bases, the quotiented multiplication operators $\bar{M}_\lambda$ and $\bar{M}_\mu$ respectively.
\item \label{item2} if $A(\lambda,\mu)$ is nonsingular, then  $W(\lambda,\mu)\colonequals  \mu W_0-\lambda W_1$, or more in general $\mu\bar{M}_\lambda-\lambda\bar{M}_\mu$ (for any choice of the bases) is a strong linearization.
\end{enumerate}
\end{theorem}
We give here only a proof of Part~\ref{item1}, and defer the one of Part~\ref{item2} to the next section, as it is immediate once we set up the right machinery.
\begin{proof}[Proof (of Part~\ref{item1} only)]
For each degree $k\in\mathbb{N}$, we choose
\[
 \mu^k e_1, \mu^k e_2,\dots, \mu^k e_n, \mu^{k-1}\lambda e_1,\dots,\mu^{k-1}\lambda e_n, \mu^{k-2}\lambda^2 e_1\dots, \lambda^k e_n
\]
as a basis of $\mathbb{C}^n_k[\lambda,\mu]$. Then, it is easy to represent $(A)$ and the non-quotiented multiplication operators
\begin{gather*}
 M_\lambda, M_\mu : \mathbb{C}^n_{d-1}[\lambda,\mu] \to \mathbb{C}^n_{d}[\lambda,\mu]\\
 M_\lambda: v(\lambda,\mu) \mapsto \lambda v(\lambda,\mu),\\
 M_\mu: v(\lambda,\mu) \mapsto \mu v(\lambda,\mu);
\end{gather*}
namely,
\begin{align*}
(A)=&\operatorname{Im}\col{A}, &
 M_\lambda=&\m{0_{n,dn}\\I_{dn}},
 &
 M_\mu=&\m{I_{dn}\\0_{n,dn}}.
\end{align*}
We now need a basis of the quotient space $\mathbb{C}^n_{d}[\lambda,\mu] / (A)$. We know from elementary linear algebra that if we complete a basis of $(A)$ to a basis of $\mathbb{C}^n_{d}[\lambda,\mu]$ in any way with $f_{n+1},f_{n+2},\dots,f_{(d+1)n}$, then the vectors $f_i+(A)$ form a basis of the quotient space.

Lemma~\ref{bcl} gives us $V\in \mathbb{C}^{n,dn}$ and $B\in\mathbb{C}^{dn,n}$ such that
\[
 \m{V\\W}\m{\col{A} & B}=I;
\]
in particular, the columns of $B$ complete those of $\col{A}$ to a basis and therefore induce a basis of the quotient space. In particular, since
\begin{align*}
v+(A)=&\left(\col{A} V+BW \right)v+(A) =BWv+(A),
\end{align*}
then $Wv$ gives the coordinates of $v(\lambda,\mu)+(A)$ in the chosen quotient space basis. In particular,
\begin{equation}\label{explWlin}
\begin{aligned}
\lambda \bar{M}_\mu - \mu \bar{M}_\lambda=&
 W\left(\lambda\m{I_{dn}\\0_{n,dn}} -\mu \m{0_{n,dn}\\I_{dn}} \right).
\end{aligned}
 \qedhere
\end{equation}
\end{proof}

\section{Dual linearizations}
We introduce in this section a general technique to derive new linearizations from existing ones.
\begin{theorem}[Left duality]\label{thm:duallin}
 Let $L(\lambda,\mu)\colonequals \mu L_0-\lambda L_1\in \mathbb{C}_1^{N,N}[\lambda,\mu]$ be a nonsingular matrix pencil, and let
$M(\lambda,\mu)\colonequals \mu M_0-\lambda M_1\in \mathbb{C}_1^{N,N}[\lambda,\mu]$ be another matrix pencil of the same size such that
 \begin{enumerate}
  \item $M_1 L_0 = M_0 L_1$;
  \item $\row{M}$ has full row rank $N$.
 \end{enumerate}
Then, $M(\lambda,\mu)$ is strongly equivalent to $L(\lambda,\mu)$. In particular, if the latter is a (strong) linearization of a matrix polynomial, then the former is a (strong) linearization as well.
\end{theorem}
Notice that the first condition can be rewritten in the more appealing form
\begin{align*}
 \row{M}\mathcal{J}\col{L}=&0, & \mathcal{J}\colonequals\m{0 & I_N\\-I_N &0}.
\end{align*}
\begin{proof}
We first prove that 
We apply Lemma~\ref{bcl} to $\row{M}$ and $\mathcal{J}\col{L}$, obtaining
\[
 \m{V_0 & V_1\\M_0 & -M_1}\m{L_1 & B_1\\L_0 & B_0}=I_{2n}.
\]
Let us choose $(\alpha,\beta)\in\mathbb{C}^2$, normalized with $\abs{\alpha}^2+\abs{\beta}^2=1$, such that $L(\alpha,\beta)$ is nonsingular; then, we have the following identity, where $\ast$ denotes blocks that are not of interest here:
\begin{align*}
 \m{-L(\alpha,\beta) & \ast \\ \ast & \ast}=&\m{\alpha I_N & -\beta I_N\\ \bar{\beta}I_N & \bar{\alpha}I_N}\m{L_1 & B_1\\L_0 & B_0}\\=&\left(\m{V_0 & V_1\\M_0 & -M_1}\m{\bar{\alpha}I_N & \beta I_N\\ -\bar{\beta} I_N & \alpha I_N}\right)^{-1}=\m{\ast & \ast \\ \ast & M(\alpha,\beta)}^{-1}.
\end{align*}
Standard Schur complement theory shows that $M(\alpha,\beta)$ is the inverse of the Schur complement of $-L(\alpha,\beta)$, and thus the former is nonsingular whenever the latter is so.

We could probably extend this argument to prove that the Jordan structures of the two pencils coincide; however, we switch to a different one that reveals more clearly the form of the strong equivalence relation between the two pencils. Let $(\alpha,\beta)$ be so that both $M(\alpha,\beta)$ and
$L(\alpha,\beta)$ are nonsingular; thanks to the first condition, the identity
\[
 M(\alpha,\beta)L(\lambda,\mu)=M(\lambda,\mu)L(\alpha,\beta)
\]
holds, and thus
\begin{equation}\label{explicitconj}
 M(\alpha,\beta)^{-1} M(\lambda,\mu)L(\alpha,\beta) = L(\lambda,\mu). 
\end{equation}
\end{proof}
If $M(\lambda,\mu)$ satisfies the hypotheses of Theorem~\ref{thm:duallin}, we say that it is a \emph{left dual} of $L(\lambda,\mu)$. Notice that a left dual is uniquely determined up to left multiplication by a nonsingular scalar matrix; in particular, its right eigenvectors are well-defined.
\begin{corollary}[Right duality]
The same results hold if we replace the two conditions in Theorem~\ref{thm:duallin} with
\begin{enumerate}
\item $L_0 M_1 = L_1 M_0$;
  \item $\col{M}$ has full column rank $N$.
\end{enumerate}
\end{corollary}
In this case, we say that $M(\lambda,\mu)$, uniquely determined up to right multiplication by a nonsingular scalar matrix, is the \emph{right dual} of $L(\lambda,\mu)$.

With Theorem~\ref{thm:duallin} in place, the proof of Part~\ref{item2} of Theorem~\ref{thm:multop} is a one-liner.
\begin{proof}[Proof of Theorem \ref{thm:multop}, Part~\ref{item2}]
$W_1 C_0=W_0 C_1$, thus $W(\lambda,\mu)$ is the left dual of a companion form!
\end{proof}

\section{Constructing duals}
The relation $M_1 L_0=M_0 L_1$ has been studied estensively in the context of pencil arithmetic and inverse-free matrix iterative algorithms \cite{Mal89,BaiDG97,DemDH07,BenB06,ChuFL05}. Two main techniques exist for constructing $M_0, M_1 \in \mathbb{C}^{m,m}$ starting from $L_0,L_1 \in \mathbb{C}^{m,m}$.
\begin{description}
 \item[QR factorization] \cite{BenB06,Ben97} Construct the QR factorization
 \[
  \begin{bmatrix} L_0 \\ L_1 \end{bmatrix}
  =
    \begin{bmatrix}
   Q_{11} & Q_{12}\\
   Q_{21} & Q_{22}
  \end{bmatrix}
\begin{bmatrix}
           R \\ 0
          \end{bmatrix}
 \]
and take $M_1=Q_{12}^*$, $M_0=-Q_{22}^*$. In practice, a QRP factorization should be used, since $\col{A}$ being close-to-rank-deficient is a concern here.

\begin{example}\label{WQR}
 Though formally it is a different problem, the same strategy works for selecting a good $W$ to use in \eqref{defws}. Construct a QR factorization
 \[
 \col{A}=
  \begin{bmatrix}
   Q_{00} & Q_{01} & \dots & Q_{0d}\\
   Q_{10} & Q_{11} & \dots & Q_{1d}\\
   \vdots & \vdots & \ddots & \vdots\\
   Q_{d0} & Q_{d1} & \dots & Q_{dd}
  \end{bmatrix}
  \begin{bmatrix}
   R \\ 0 \\ \vdots \\0
  \end{bmatrix}
 \]
 (all the blocks are $n\times n$) and take
 \[
  W=\begin{bmatrix}
   Q_{01} & Q_{02} & \dots & Q_{0d}\\
   Q_{11} & Q_{12} & \dots & Q_{1d}\\
   \vdots & \vdots & \ddots & \vdots\\
   Q_{d1} & Q_{d2} & \dots & Q_{dd}
    \end{bmatrix}^*.
 \]
\end{example}

\item[Enforcing an identity block] \cite{ChuFL05,MehP10_ppt} Suppose that the identity matrix is a submatrix of $\begin{bmatrix}L_0 \\ L_1\end{bmatrix}$. Then, we can select a permutation matrix $\Pi \in \mathbb{C}^{2m,2m}$ and $X\in\mathbb{C}^{m,m}$ such that
\[
 \begin{bmatrix}L_0\\L_1\end{bmatrix} = \Pi \begin{bmatrix}I\\X\end{bmatrix}.
\]
Then, the identity
\[
 0=\left(  \begin{bmatrix}
  -X & I
 \end{bmatrix}
 \Pi^{-1}\right)
 \Pi
 \begin{bmatrix}
  I\\X
 \end{bmatrix}
\]
holds, and thus we can choose
\[
 \begin{bmatrix}
  M_1 & -M_0
 \end{bmatrix}
 =
 \begin{bmatrix}
  -X & I
 \end{bmatrix}
 \Pi^{-1}.
\]
Generalizing slightly, if a $m\times m$ submatrix $Y$ of $\begin{bmatrix}L_0\\L_1\end{bmatrix}$ is known to be nonsingular, we have
\[
 \begin{bmatrix}L_0\\L_1\end{bmatrix}= \Pi \begin{bmatrix}Y\\Z\end{bmatrix} = \Pi \begin{bmatrix}I\\ZY^{-1}\end{bmatrix}Y
\]
and thus
\[
 \begin{bmatrix}
  M_1 & -M_0
 \end{bmatrix}
 =
 \begin{bmatrix}
  -ZY^{-1} & I
 \end{bmatrix}
 \Pi^{-1}. 
\]
\begin{example}
The simplest nontrivial pencil in the Fiedler family \cite[Example~2.5]{AntV04} is
\begin{equation}\label{thefiedler}
 \lambda
 \begin{bmatrix}
  0 & A_0 & 0\\
  I & A_1 & 0\\
  0 & 0 & I
 \end{bmatrix}
-\mu
\begin{bmatrix}
 I & 0 & 0\\
 0 & -A_2 & -A_3\\
 0 & I & 0
\end{bmatrix}.
\end{equation}
We choose $\Pi$ as the block permutation $(1,2,3,4,5,6) \mapsto (1,3,6,4,5,2)$, so that the leading block equals the identity matrix. Then, the method yields
\[
 \begin{bmatrix}
  M_1 & -M_0
 \end{bmatrix}
=\begin{bmatrix}
  0 & 0 & -A_0 & I & 0 & 0\\
  -I & 0 & -A_1 & 0 & I & 0\\
  0 & I & A_2 & 0 & 0 & A_3
 \end{bmatrix},
\]
which is, up to sign changes, a companion form. Conversely, if one did not know about Fiedler pencils at all, they could derive  \eqref{thefiedler} with the same technique as the right dual of this companion form.
\end{example}
\end{description}

\section{Relationship with other known linearizations}
\subsection{Vector spaces of linearizations}
A linearization $L(\lambda,\mu)\colonequals \mu L_0 - \lambda L_1$ is a right dual of $W(\lambda,\mu)$ if and only if
\[
 0=W_0 L_1 - W_1 L_0 = W\left(\m{0_{n,dn}\\I_{dn}}L_1-\m{I_{dn}\\0_{n,dn}}L_0\right).
\]
Since $\ker W=\col{A}$, this holds if and only if
\begin{equation}\label{rowshifted2}
 \m{0_{n,dn}\\I_{dn}}L_1-\m{I_{dn}\\0_{n,dn}}L_0 = \m{A_0\\ A_1\\ \vdots \\ A_d}Z
\end{equation}
for some $Z\in\mathbb{C}^{n,dn}$. This generalizes \eqref{rowshifted}, and thus all pencils in $\mathbb{L}_2$ are contained in the right dual of $W(\lambda,\mu)$. Analogously, $\mathbb{L}_1$ can be obtained as left dual of another companion form.

In \cite{MacMMM06}, it is proved that any nonsingular $L(\lambda,\mu)$ satisfying \eqref{rowshifted} is a strong linearization. Using Theorem~\ref{thm:duallin}, we can replace nonsingularity with a weaker condition.
\begin{corollary}
 Let $A(\lambda,\mu)$ be a nonsingular matrix polynomial, and $L(\lambda,\mu)$ be any matrix pencil satisfying \eqref{rowshifted2}. Then, $L$ is a strong linearization of $A(\lambda,\mu)$ if and only if $\col{L}$ has full column rank.
\end{corollary}

\subsection{Fiedler pencils}
We have already shown in Example~\ref{thefiedler} that pencil duality can be used to obtain naturally at least one Fiedler pencil starting from the companion form. This is no accident: the approach used in \cite{AntV04} to prove that the Fiedler pencils are linearizations can be reinterpreted in our setting as a sequence of duality operations that transforms a generic Fiedler pencil $F(\lambda,\mu)$ into $C(\lambda,\mu)$. Namely, one can recognize that \cite[Lemma~2.2]{AntV04} there is essentially a special case of our Theorem~\ref{thm:duallin}. In particular, the result proved here provides a correct proof of its part (b), which is obtained in \cite{AntV04} using the false implication $M(\lambda,\mu)$ singular $\Rightarrow$ $\row{M}$ is rank-deficient.

\section{Eigenvector recovery}
It follows from \eqref{explicitconj} that for each $(\alpha,\beta)$ such that $A(\alpha,\beta)$ is nonsingular
\begin{equation}\label{checkeig}
\begin{aligned}
 \check{y}^*\colonequals&\hat{y}^*W(\alpha,\beta)^{-1},\\
 \check{x}\colonequals&C(\alpha,\beta)\hat{x}
\end{aligned}
\end{equation}
are left and right eigenvectors of $W(\lambda,\mu)$. Different choices of $(\alpha,\beta)$ in \eqref{checkeig} yield the same $\check{x}$ and $\check{y}$, up to normalization; this is proved by the following computations, which highlight the role of the expression $\beta\lambda-\alpha\mu$.
\begin{align*}
 &\m{\mu^d y^* & \mu^{d-1}\lambda y^* & \dotsm & \lambda^d y^*}B(\beta W_0-\alpha W_1)\\
 =&\m{\mu^d y^* & \mu^{d-1}\lambda y^* & \dotsm & \lambda^d y^*} BW \left(\beta\m{0_{n\times dn}\\ I_{dn}}-\alpha\m{I_{dn}\\0_{n\times dn}}\right)\\
 =&\m{\mu^d y^* & \mu^{d-1}\lambda y^* & \dotsm & \lambda^d y^*} (I_{(d+1)n}-AV) \left(\beta\m{0_{n\times dn}\\ I_{dn}}-\alpha\m{I_{dn}\\0_{n\times dn}}\right)\\
 =&\m{\mu^d y^* & \mu^{d-1}\lambda y^* & \dotsm & \lambda^d y^*} I_{(d+1)n} \left(\beta\m{0_{n\times dn}\\ I_{dn}}-\alpha\m{I_{dn}\\0_{n\times dn}}\right)\\
 =&(\beta\lambda-\alpha\mu) \hat{y}^*
\end{align*}
and thus
\[
 \check{y}^*=\frac{1}{\beta\lambda-\alpha\mu}\m{\mu^d y^* & \mu^{d-1}\lambda y^* & \dotsm & \lambda^d y^*}B.
\]
For the right eigenvalues, one can verify that in the case $\lambda\neq 0,\mu\neq 0$
\begin{equation}\label{checkx}
\begin{aligned}
 \check{x}=& C(\lambda,\mu) \hat{x} =\\
=& -(\beta\lambda-\alpha\mu)
\frac{1}{\mu}
\m{
0 & 1 & \frac{\lambda}{\mu} & \frac{\lambda^2}{\mu^2} &  \dotsm & \frac{\lambda^{d-1}}{\mu^{d-1}}\\
  & \ddots & \ddots & \ddots & \ddots & \vdots \\
  & & 0 & 1 & \frac{\lambda}{\mu} & \frac{\lambda^2}{\mu^2}\\
  & & & 0 & 1 &\frac{\lambda}{\mu}\\
   & & & & 0 & 1
}
\m{A_0\\A_1\\ A_2 \\ \vdots \\ A_d}x\\
=&(\beta\lambda-\alpha\mu)
\frac{1}{\lambda}
\m{
1 & 0 \\
\frac{\mu}{\lambda} & 1 & 0\\
\frac{\mu^2}{\lambda^2} & \frac{\mu}{\lambda} & 1 & 0\\
\vdots & \ddots & \ddots & \ddots  & \ddots\\
\frac{\mu^{d-1}}{\lambda^{d-1}} & \dotsm & \frac{\mu^2}{\lambda^2} & \frac{\mu}{\lambda} & 1 & 0\\
}
\m{A_0\\A_1\\ A_2 \\ \vdots \\ A_d}x,\\
\end{aligned}
\end{equation}
and when $\lambda=0$ or $\mu=0$ the version without vanishing denominators is valid as well.

\section{Conditioning}
In this section, we wish to estimate the conditioning of the linearization $W(\lambda,\mu)$ proposed here, using the formula from \cite{HigMT06}
\[
 \kappa_W(\lambda,\mu)=\left(\abs{\lambda}^2\norm{W_1}^2+\abs{\mu}^2\norm{W_0}^2\right)^{1/2}\frac{\norm{\check{x}}\norm{\check{y}}}{\abs{\check{y}^*\left(\bar{\mu} W_1+\bar{\lambda} W_0\right)\check{x}}}.
\]
We estimate separately the factors in the formula. Using an argument similar to \cite[page~1020]{HigMT06}, we rewrite the denominator as
\begin{align*}
 \check{y}^*\val{\left(\bar{\mu} W_1+\bar{\lambda} W_0\right)}{(\lambda,\mu)} \check{x} = &\hat{y}^*\left(\bar{\mu} C_1+\bar{\lambda} C_0\right)\hat{x}=y^*\val{\left(\bar{\mu} D_\lambda P-\bar{\lambda} D_\mu P\right)}{(\lambda,\mu)}x.
\end{align*}

Normalize now so that $\abs{\lambda}^2+\abs{\mu}^2=1$, and notice that we are always in one of the two cases
\[
 \begin{cases}
  \abs{\lambda}\leq\abs{\mu} \Rightarrow  \abs{\frac{\lambda}{\mu}}\leq 1,  \abs{\frac{1}{\mu}}\leq \sqrt{2},\\
  \abs{\mu}\leq\abs{\lambda} \Rightarrow  \abs{\frac{\mu}{\lambda}}\leq 1,  \abs{\frac{1}{\lambda}}\leq \sqrt{2}.\\
 \end{cases}
\]
In both cases, we may estimate in \eqref{checkx}
\[
 \norm{\check{x}} \leq \abs{\beta\lambda-\alpha\mu}\sqrt{2} T(d)\norm{\col{A}},
\]
where
\[
 T(d)=\norm{\begin{bmatrix}1 & 1 & \dotsm & 1\\ & 1 & \ddots & \vdots \\ & & \ddots & 1\\ & & & 1\end{bmatrix}}.
\]
An elementary linear algebra argument \cite{mathov} yields
\begin{equation}\label{mathov}
 T(d)=\frac{1}{2\sin\frac{\pi}{2d+2}} \approx \frac{2d}{\pi},
\end{equation}
so $T(d)$ grows linearly with $d$.

An easier estimate holds for $\check{y}$:
\[
 \norm{\check{y}} \leq \frac{1}{\abs{\beta\lambda-\alpha\mu}}\norm{y} \norm{\abs{\Lambda}_{\lambda,\mu}}\norm{B},
\]
with
\[
\abs{\Lambda}_{\lambda,\mu}\colonequals \begin{bmatrix}\abs{\lambda}^{d} & \abs{\lambda}^{d-1}\abs{\mu} & \dotsm & \abs{\mu}^{d}\end{bmatrix}.
\]
Putting everything together, we obtain
\[
 \kappa_W(\lambda,\mu)\leq \left(\abs{\lambda}^2\norm{W_1}^2+\abs{\mu}^2\norm{W_0}^2\right)^{1/2}\sqrt{2}T(d)\norm{\abs{\Lambda}_{\lambda,\mu}}\norm{\col{A}}\norm{B}\frac{\norm{x}\norm{y}}{\abs{y^*\val{\left(\bar{\mu} D_\lambda P-\bar{\lambda} D_\mu P\right)}{(\lambda,\mu)}x}}.
\]
The bound gets simpler if we choose $W$ with orthonormal columns as suggested in Example~\ref{WQR}; in this case $\norm{W_0}\leq 1$, $\norm{W_1}\leq 1$, $\norm{B}\leq 1$, and thus
\[
 \kappa_W(\lambda,\mu)\leq \sqrt{2}T(d)\norm{\abs{\Lambda}_{\lambda,\mu}}\norm{\col{A}}\frac{\norm{x}\norm{y}}{\abs{y^*\left(\bar{\mu} \val{D_\lambda P-\bar{\lambda} D_\mu P\right)}{(\lambda,\mu)}x}}.
\]
Apart from the moderate factor $\sqrt{2}T(d)$, the only difference with the condition number of the polynomial eigenvalue problem is that $\norm{\abs{\Lambda}_{\lambda,\mu}}\norm{\col{A}}$ is replaced by
\[
\left(\sum_{i=0}^d\abs{\lambda}^{2i}\abs{\mu}^{2(d-i)}\norm{A_i}\right)^{1/2} = \abs{\Lambda}_{\lambda,\mu}\m{\norm{A_0}\\ \norm{A_1}\\ \vdots \\ \norm{A_d}}
\]
there. This reveals that the only problematic case is the one in which the two vectors appearing in this expression are nearly orthogonal. In the quadratic case, if proper scaling is performed, this can happen only if $\norm{A_1}\gg \norm{A_0}+\norm{A_2}$ and the eigenvalue is close to either $0$ or $+\infty$. Incidentally, \cite[Theorem~4.3]{HigMT06} shows that this is the problematic case for the $\mathbb{DL}$ linearizations as well.

In fact, for $W(\lambda,\mu)$ the conditioning of the linearization is only part of the story, since getting from the original data (coefficients of $A(\lambda,\mu)$) to $W(\lambda,\mu)$ requires a QRP factorization. Multiplying the conditioning of the QRP factorization by this condition number is likely to overestimate the error. Therefore, in the next section we report numerical results to obtain an practical assessment of the accuracy of the eigenvalues.

\section{Numerical experiments}
Our primary goal in this section is comparing the quality of the eigenvalues computed from $W(\lambda,\mu)$ with that of two common choices for unstructured problems. One is the companion form \eqref{compform}, which is used in Matlab's \texttt{polyeig}; the other is the strategy suggested in \cite{HigMT06}, namely,
\begin{itemize}
 \item if for an eigenvalue $\abs{x}>1$ (i.e., $\abs{\lambda}>\abs{\mu}$), then use the $\mathbb{DL}$ linearization with $v=e_1$
 \item otherwise ($\abs{x}\leq 1$, i.e., $\abs{\lambda}\leq \abs{\mu}$), use the $\mathbb{DL}$ linearization with $v=e_d$.
\end{itemize}
Their theoretical result proving a bound on the eigenvalue condition numbers is valid only in the case in which there are no $0$ or $\infty$ eigenvalues, i.e., $A_0$ and $A_d$ are nonsingular; we restrict the comparison to this case for now.

As stated before, comparing only the condition numbers of the linearizations would ignore the errors deriving from the QRP factorization; therefore, we chose to compare directly the forward errors on the eigenvalues. The same comparison is used in \cite{HigMT06}.

A set of reference values is obtained by using Matlab's variable precision arithmetic. Unfortunately, as of version R2010b, none of the customary methods for computing generalized eigenvalues (\texttt{polyeig}, \texttt{qz} and the two-argument version of \texttt{eig}) can deal with variable precision inputs; therefore, we have to settle for an inferior approach in computing the reference values. Namely, we compute them using \texttt{eig(G0/G1)}, where $G_0-xG_1$ is the $v=e_1$ linearization. Both inversion and eigenvalue computation are performed with the default (for VPA) value of $32$ significant digits. Incidentally, this is another reason to restrict our comparison to problems without infinite eigenvalues.

Then, using the two-argument \texttt{eig}, we compute the eigenvalues of the pencil obtained by linearization with all the four linearizations under examination, and compare them to the reference values using the angular distance ($d(v,w)=$\texttt{subspace([v;1],[w;1])}).

The benchmark problems used are chosen among the regular quadratic eigenvalue problems ($d=2$) in the collection NLEVP \cite{nlevp}, preprocessed with the Fan--Lin--Van Dooren scaling. As argued before, we restrict to those problems for which the condition numbers of $A_0$ and $A_2$ are smaller than $10^{10}$; moreover, the slowness of variable precision arithmetic makes it impractical to deal with large problems, therefore we only consider the problems with $n<150$. That leaves us with 17 problems, of which we have to further exclude problem \texttt{sign1} as the computation of the reference eigenvalues did not converge due to a VPA error. We report the angular errors for the remaining test problems in the figures.

\newcommand{\myfigurewithothertitle}[4]{
\begin{tikzpicture}
\begin{axis}[#3,ymode=log,xlabel={eigenvalue \#},title={\texttt{#2}}]
\pgfplotstableread{#1.dat}{\currenttable}
\pgfplotstablecreatecol[create col/expr={\pgfplotstablerow + 1}]{row number}{\currenttable}
 \addplot[color=blue,only marks,mark=x,mark size=0.3em] table[x=row number, y index=0]{\currenttable};\addlegendentry{$C(\lambda,\mu)$};
 \addplot[color=red,only marks,mark=+,mark size=0.3em] table[x=row number, y index=1]{\currenttable};\addlegendentry{$W(\lambda,\mu)$};
 \addplot[color=green,only marks,mark=asterisk,mark size=0.3em] table[x=row number, y index=2]{\currenttable};\addlegendentry{$v=e_1$}
 \addplot[color=brown,only marks,mark=star,mark size=0.3em] table[x=row number, y index=3]{\currenttable};\addlegendentry{$v=e_d$}
 \addplot[color=black,only marks,mark=o,mark size=0.3em] table[x=row number, y index=4]{\currenttable};\addlegendentry{[HMT]}
 #4
\end{axis}
\end{tikzpicture}
}
\newcommand{\myfigure}[3]{\myfigurewithothertitle{#1}{#1}{#2}{#3}}

\begin{figure}\centering
\myfigurewithothertitle{acoustic_wave_1d}{acoustic\_wave\_1d}{width=0.5\textwidth,legend style={at={(0,1)},anchor=north west}}{\legend{}}
\myfigure{bicycle}{width=0.5\textwidth,legend style={at={(0,1)},anchor=north west}}{}
\end{figure}

\begin{figure}\centering
\myfigurewithothertitle{gen_hyper2}{gen\_hyper2}{width=0.5\textwidth,legend style={at={(1,1)},anchor=north west}}{\legend{}}
\myfigurewithothertitle{metal_strip}{metal\_strip}{width=0.5\textwidth,legend style={at={(1,1)},anchor=north west}}{\legend{}}
\end{figure}
\begin{figure}\centering
\myfigurewithothertitle{power_plant}{power\_plant}{width=0.5\textwidth,legend style={at={(0,1)},anchor=north west}}{\legend{}}
\myfigure{qep2}{width=0.5\textwidth,legend style={at={(1,1)},anchor=north east}}{}
\end{figure}

\begin{figure}\centering
\myfigure{sleeper}{width=0.5\textwidth,legend style={at={(0,1)},anchor=north west}}{\legend{}}
\myfigure{spring}{width=0.5\textwidth,legend style={at={(0,1)},anchor=north west}}{\legend{}}
\end{figure}
\begin{figure}\centering
\myfigure{wing}{width=0.5\textwidth,legend style={at={(0,1)},anchor=north west}}{\legend{}}
\myfigure{wiresaw1}{width=0.5\textwidth,legend style={at={(1,1)},anchor=north west}}{\legend{}}
\end{figure}
\begin{figure}\centering
\myfigure{wiresaw2}{width=0.5\textwidth,legend style={at={(1,1)},anchor=north west}}{\legend{}}
\end{figure}
\begin{figure}\centering
\myfigurewithothertitle{acoustic_wave_2d}{acoustic\_wave\_2d}{width=\textwidth,height=0.48\textheight,legend style={at={(1,1)},anchor=north west}}{\legend{}}
\myfigurewithothertitle{cd_player}{cd\_player}{width=\textwidth,height=0.48\textheight,legend style={at={(0,1)},anchor=north west}}{}
\end{figure}
\begin{figure}\centering
\myfigure{dirac}{width=\textwidth,height=0.48\textheight,legend style={at={(1,0)},anchor=south east}}{\legend{}}
\myfigure{hospital}{width=\textwidth,height=0.48\textheight,legend style={at={(1,1)},anchor=north west}}{\legend{}}
\end{figure}
\begin{figure}\centering
\myfigure{sign2}{width=\textwidth,height=0.48\textheight,legend style={at={(1,0)},anchor=south east}}{\legend{}}
\end{figure}

The results are generally good, on par or better than the other alternatives. In problems \texttt{bicycle} and \texttt{metal\_strip}, the results are slightly inferior. On the other hand, in problems \texttt{power\_plant} and \texttt{wing} the accuracy of the new method is higher, and minor benefits on some eigenvalues can be identified in other problems. Moreover, notice that the strategy in [HMT] requires solving two separate eigenproblems, thus producing similar results with half of the computational work is an interesting result in itself.

We tested the new linearization $W(\lambda,\mu)$ on problems with several zero and/or infinite eigenvalues as well, without however obtaining good computational results. The reasons for this failure and the performance of this method in presence of zero and infinite eigenvalues are still under investigation, as it is not easy for us to identify the impact of the eigenvalue solver and of the various parts of the algorithm.

\section{Acknowledgements}
The author would like to thank D. Steven Mackey, Christian Mehl and Volker Mehrmann for several useful and illuminating discussions on linearizations. A proof of the relation \eqref{mathov} was suggested by Noam D.~Elkies on the excellent website MathOverflow, in response to an inquiry of the author \cite{mathov}.

\bibliographystyle{plain}
\bibliography{qrlin}

\end{document}